\documentclass[a4paper]{amsart}

\usepackage{amsmath,amssymb,amsthm,mathrsfs}
\usepackage{fullpage}
\usepackage[alphabetic]{amsrefs}
\usepackage[all]{xy}
\newdir{ >}{{}*!/-5pt/@{>}} 
\SelectTips{cm}{} 

\title{The weight and rank filtrations}
\author{John Rognes}
\address{Department of Mathematics, University of Oslo, Norway}
\email{rognes@math.uio.no}
\date{October 28th 2016}

\newtheorem{theorem}{Theorem}[section]
\newtheorem{proposition}[theorem]{Proposition}

\newtheorem{corollary}[theorem]{Corollary}
\newtheorem{conjecture}[theorem]{Conjecture}
\theoremstyle{definition}
\newtheorem{definition}[theorem]{Definition}
\theoremstyle{remark}
\newtheorem{example}[theorem]{Example}
\newtheorem{remark}[theorem]{Remark}

\DeclareMathOperator*{\colim}{colim}
\DeclareMathOperator{\Ext}{Ext}
\DeclareMathOperator{\Sub}{Sub}
\DeclareMathOperator{\Hom}{Hom}
\DeclareMathOperator{\Lie}{Lie}
\DeclareMathOperator{\St}{St}
\DeclareMathOperator{\im}{im}
\DeclareMathOperator{\Spec}{Spec}
\newcommand{\bfA}{\mathbf{A}}
\newcommand{\A}{\mathbb{A}}
\newcommand{\B}{\mathcal{B}}
\newcommand{\cof}{\rightarrowtail}
\newcommand{\D}{\mathbf{D}}
\newcommand{\F}{\mathbb{F}}
\newcommand{\K}{\mathbf{K}}

\newcommand{\Q}{\mathbb{Q}}
\newcommand{\X}{\mathbf{X}}
\newcommand{\Z}{\mathbb{Z}}
\renewcommand{\:}{\colon}
\renewcommand{\O}{\mathcal{O}}
\renewcommand{\P}{\mathscr{P}}
\newcommand{\bbP}{\mathbb{P}}
\renewcommand{\S}{\mathbf{S}}

\begin{document}
\maketitle

\begin{abstract}
We compare the weight and stable rank filtrations of algebraic $K$-theory,
and relate the Beilinson--Soul{\'e} vanishing conjecture to the author's
connectivity conjecture.
\end{abstract}

\section{Introduction}

Let $\F$ be a field.
For $j\ge0$ let $K_j(\F)$ be the (higher) algebraic $K$-groups of $\F$
[Quillen (1970)].

There is a decreasing \emph{weight filtration}
$$
K_j(\F) \supset F^1 K_j(\F) \supset \dots \supset F^w K_j(\F)
	\supset \dots \supset F^j K_j(\F) \supset 0
$$
associated to the $\lambda$- and Adams-operations on $K_j(\F)$
[Grothendieck, Quillen/Hiller (1981)].

There is also an increasing \emph{stable rank filtration}
$$
0 \subset F_1 K_j(\F) \subset \dots \subset F_r K_j(\F)
	\subset \dots \subset F_j K_j(\F) \subset K_j(\F) \,.
$$
given by the image filtration
$$
F_r K_j(\F) = \im(\pi_j F_r \K(\F) \to \pi_j \K(\F))
$$
associated to a sequence of spectra
$$
* \to F_1 \K(\F) \to \dots \to F_r \K(\F) \to \dots \to \K(\F)
$$
called the \emph{spectrum level rank filtration}
[Rognes (1992)].

\begin{conjecture}[Beilinson--Soul{\'e}]
$$
F^w K_j(\F) = K_j(\F)
$$
for $2w \le j+1$.
\end{conjecture}

A stronger form asserts this equality also for $2w = j+2$ when $j>0$.
The conjecture is known with finite coefficients, so it suffices to
verify it rationally, i.e., after tensoring over $\Z$ with $\Q$.
We write $A_\Q = A \otimes_{\Z} \Q$.

\begin{conjecture}[Connectivity]
$$
F_r K_j(\F) = K_j(\F)
$$
for $2r \ge j+1$.
\end{conjecture}

A stronger form asserts rational equality also for $2r = j$ when $j>0$.

\begin{conjecture}[Stable Rank]
$$
F^w K_j(\F)_\Q = F_r K_j(\F)_\Q
$$
for $w+r = j+1$.
\end{conjecture}

I will provide evidence for the connectivity and stable rank conjectures,
which, if true, will imply the Beilinson--Soul{\'e} vanishing conjecture.

\section{Algebraic $K$-theory}

Let $\P(\F)$ be the category of finitely generated projective $\F$-modules,
i.e., the category of finite-dimensional $\F$-vector spaces.

The dimension $\dim V$ of an object $V$ defines an additive invariant
in $K_0(\F) \cong \Z$.

The determinant $\det A$ of an automorphism $A \: V \to V$ defines an
additive invariant in $K_1(\F) \cong \F^\times$.

The classifying space $|i\P(\F)|$ of the subcategory $i\P(\F)$
of isomorphisms in $\P(\F)$ is built with one $q$-simplex $\Delta^q$
for each chain
$$
V_0 \overset{\cong}\to V_1 \overset{\cong}\to
	\dots \overset{\cong}\to V_q
$$
of $q$ composable morphisms in $i\P(\F)$.  The inclusion of the full
subcategory generated by the objects $\F^r$, with automorphism groups
$GL_r(\F)$, induces an equivalence
$$
\coprod_{r\ge0} BGL_r(\F) \simeq |i\P(\F)| \,.
$$

Direct sum of vector spaces, $(V, W) \mapsto V \oplus W$, makes
$|i\P(\F)|$ a (coherently homotopy commutative) topological monoid.
To group complete $\pi_0 \cong \{r\ge0\}$ and $\pi_1 \cong GL_r(\F)$
(for a suitable choice of base point), we map $K(\F)_0 = |i\P(\F)|$ to a
grouplike (coherently homotopy commutative) topological monoid, i.e.,
a (infinite) loop space $\Omega K(\F)_1 = \Omega |i S_\bullet \P(\F)|$,
to be defined below.
The map $K(\F)_0 \to \Omega K(\F)_1$ is a group completion.

\begin{definition}
$K_j(\F) = \pi_j \Omega K(\F)_1 = \pi_{j+1} K(\F)_1$, where
$K(\F)_1 = |i S_\bullet \P(\F)|$ is given by Waldhausen's
$S_\bullet$-construction.
\end{definition}

\section{The algebraic $K$-theory spectrum}

Waldhausen's $S_\bullet$-construction applied to $\P(F)$ is a simplicial
category
$$
[q] \mapsto i S_q \P(F) \,.
$$
The category in degree~$q$ has objects the sequences of
injective homomorphisms
$$
0 = V_0 \cof V_1 \cof \dots \cof V_q
$$
in $\P(F)$,
together with compatible choices of quotients $V_j/V_i$ for $0 \le i \le j
\le q$.  Morphisms are vertical isomorphisms
$$
\xymatrix{
0 = V_0 \ar@{ >->}[r] \ar[d]^{\cong}
& V_1 \ar@{ >->}[r] \ar[d]^{\cong}
& \dots \ar@{ >->}[r]
& V_q \ar[d]^{\cong} \\
0 = W_0 \ar@{ >->}[r]
& W_1 \ar@{ >->}[r]
& \dots \ar@{ >->}[r]
& W_q
}
$$
of horizontal diagrams.

The construction can be iterated $n\ge1$ times.  We define
$$
K(\F)_n = | i S^{(n)}_\bullet \P(\F) |
$$
as the classifying space of the simplicial category
$$
[q] \mapsto i S^{(n)}_q \P(\F)
$$
with objects in degree~$q$ given by $n$-dimensional cubical diagrams
$[q]^n \to \P(\F)$.  In the case $n=2$
$$
\xymatrix{
0 \ar@{=}[r] \ar@{=}[d]
& 0 \ar@{=}[r] \ar@{ >->}[d]
& \dots \ar@{=}[r]
& 0 \ar@{ >->}[d] \\
0 \ar@{ >->}[r] \ar@{=}[d]
& V_{1,1} \ar@{ >->}[r] \ar@{ >->}[d]
& \dots \ar@{ >->}[r]
& V_{1,q} \ar@{ >->}[d] \\
\vdots \ar@{=}[d]
& \vdots \ar@{ >->}[d]
& \ddots
& \vdots \ar@{ >->}[d] \\
0 \ar@{ >->}[r]
& V_{q,1} \ar@{ >->}[r]
& \dots \ar@{ >->}[r]
& V_{q,q}
}
$$
we require to have injective homomorphisms $V_{i-1,j} \to V_{i,j}$
and $V_{i,j-1} \to V_{i,j}$ and injective pushout homomorphisms
$$
V_{i-1,j} \oplus_{V_{i-1,j-1}} V_{i,j-1} \to V_{i,j} \,,
$$
for all $1 \le i,j \le q$.  For higher~$n$ there are similar conditions
for $d$-dimensional subcubes for all $1 \le d \le n$.

\begin{definition}
The algebraic $K$-theory spectrum of $\F$ is the spectrum
$$
\K(\F) = \{n \mapsto K(\F)_n = |i S^{(n)}_\bullet \P(\F) | \} \,.
$$
It is positive fibrant, in the sense that
$K(\F)_n \to \Omega K(\F)_{n+1}$ is an equivalence for each $n\ge1$.
Hence
$$
K_j(\F) = \pi_j \K(\F) = \pi_{j+n} K(\F)_n
$$
for each $n\ge1$.
\end{definition}

This constructions produces a symmetric spectrum: the group $\Sigma_n$
permutes the order of the $n$ instances of the $S_\bullet$-construction.

\section{Weight filtration}

Let $k\ge0$.  The $k$-th exterior power $V \mapsto \Lambda^k V$
induces $\lambda$-operations
$$
\lambda^k \: K_j(\F) \to K_j(\F) \,.
$$
For $V = L_1 \oplus \dots \oplus L_r$ a direct sum of lines,
$$
\Lambda^k V = \bigoplus_{1 \le i_1 < \dots < i_k \le r}
	L_{i_1} \otimes \dots \otimes L_{i_k}
$$
corresponds to the $k$-th elementary symmetric polynomial
$$
\sigma_k(x_1, \dots, x_r) = \sum_{1 \le i_1 < \dots < i_k \le r}
	x_{i_1} \dots x_{i_r} \,.
$$
The $k$-th Adams operation
$$
\psi^k \: K_j(\F) \to K_j(\F)
$$
is induced by $L_1 \oplus \dots \oplus L_r \mapsto L_1^{\otimes k}
\oplus \dots \oplus L_r^{\otimes k}$ and corresponds to the
$k$-th power sum polynomial
$$
s_k(x_1, \dots, x_r) = \sum_{i=1}^r x_i^k \,.
$$
It can thus be expressed in terms of the $\lambda$-operations.

The weight filtration $\{F^w K_j(\F)\}_{w\ge0}$ on $K_j(\F)$ is constructed
by means of the $\lambda^k$.  The Adams operations satisfy
$$
\psi^k(x) \equiv k^w x \mod F^{w+1} K_j(\F)
$$
for $x \in F^w K_j(\F)$.  Hence $\psi^k$ acts as multiplication by~$k^w$
on $F^w K_j(\F) / F^{w+1} K_j(\F)$, for all $k\ge0$.

Rationally the weight filtration splits as a direct sum of common
eigenspaces for the Adams operations.  Let
$$
K_j(\F)_\Q^{(w)} = \{x \in K_j(\F)_\Q \mid
	\text{$\psi^k(x) = k^w x$ for all $k$} \}
$$
be the \emph{weight $w$ rational eigenspace}.  Then
$$
F^w K_j(\F)_\Q = \bigoplus_{v \ge w} K_j(\F)_\Q^{(v)}
$$
is the subspace of weights $\ge w$, and
$$
\frac{F^w K_j(\F)_\Q }{ F^{w+1} K_j(\F)_\Q} \cong K_j(\F)_\Q^{(w)} \,.
$$
Soul{\'e} proved that $K_j(\F)_\Q^{(w)} = 0$ for $w < 0$ and for $w > j$,
i.e., $K_j(\F)_\Q$ only contains classes of weight $0 \le w \le j$.

\begin{example}
$\psi^k$ acts (additively) on $K_1(\F)$ by $\psi^k(x) = kx$, since
it acts (multiplicatively) on $\F^\times$ by $u \mapsto u^k$.
Hence all of $K_1(\F)_\Q = K_1(\F)_\Q^{(1)}$ has weight~$1$.
The product $x_1 \cdots x_j \in K_j(\F)$ of $j$ classes $x_1, \dots, x_j
\in K_1(\F)$ has weight~$j$, since $\psi^k(x_1 \cdots x_j) = \psi^k(x_1)
\cdots \psi^k(x_j) = (kx_1) \cdots (kx_j) = k^j x_1 \cdots x_j$.

Milnor $K$-theory $K^M_*(\F)$ is defined to be the quotient of the tensor
algebra on $F^\times$, over $\Z$, by the ideal generated
by $u \otimes (1-u)$ for $u \in \F \setminus \{0,1\}$.  It is graded
commutative, so in degree~$j$ there are surjections
$$
(\F^\times)^{\otimes j} \to \Lambda^j \F^\times \to K^M_j(\F) \,.
$$
By the Steinberg relation $\{u, 1-u\} = 0$ in $K_2(\F)$, these all map
to $K_j(\F)$, and land in the weight~$j$ eigenspace.
\end{example}

\section{Motivic cohomology}

By analogy with the Atiyah--Hirzebruch spectral sequence from singular
cohomology to topological $K$-theory for a topological space, there is
a \emph{motivic spectral sequence}
$$
E^2_{s,t}(mot) = H^{t-s}_{mot}(\F, \Z(t))
	\Longrightarrow K_{s+t}(\F) \,.
$$
It is of homological type, concentrated in the first quadrant
($s\ge0$ and $t\ge0$), and collapses rationally at the $E^2$-term
($d^r = 0$ after rationalization for $r\ge2$).
$$
\xymatrix{
t & \dots & \dots & \dots & \dots & \dots \\
4 & H^4_{mot}(\F; \Z(4))
& H^3_{mot}(\F; \Z(4))
& H^2_{mot}(\F; \Z(4))
& H^1_{mot}(\F; \Z(4)) & \dots \\
3 & H^3_{mot}(\F; \Z(3))
& H^2_{mot}(\F; \Z(3))
& H^1_{mot}(\F; \Z(3)) \ar[llu]_{d^2}
& H^0_{mot}(\F; \Z(3)) & \dots \\
2 & H^2_{mot}(\F; \Z(2))
& H^1_{mot}(\F; \Z(2))
& H^0_{mot}(\F; \Z(2))
& H^{-1}_{mot}(\F; \Z(2)) & \dots \\
1 & H^1_{mot}(\F; \Z(1))
& H^0_{mot}(\F; \Z(1))
& H^{-1}_{mot}(\F; \Z(1))
& H^{-2}_{mot}(\F; \Z(1)) & \dots \\
0 & H^0_{mot}(\F; \Z(0))
& H^{-1}_{mot}(\F; \Z(0))
& H^{-2}_{mot}(\F; \Z(0))
& H^{-3}_{mot}(\F; \Z(0)) & \dots \\
E^2_{s,t} & 0 & 1 & 2 & 3 & s
}
$$

Rationally, the motivic cohomology groups can be defined in terms of
the weight filtration.

\begin{definition}
$$
H^{t-s}_{mot}(\F, \Z(t))_\Q = H^{t-s}_{mot}(\F, \Q(t))
	= K_{s+t}(\F)_\Q^{(t)}
$$
so that
$$
H^i_{mot}(\F; \Q(w)) = K_{2w-i}(\F)_\Q^{(w)} \,.
$$
\end{definition}

\begin{remark}
These groups give the $E^2$-term of a rational motivic spectral sequence
collapsing to $K_{s+t}(\F)_\Q$.

The definition is not effective.  $K_*(\F)$ with its Adams operations
is hard to calculate and to work with.

By Soul{\'e}'s result, $H^{t-s}_{mot}(\F; \Q(t))$ can only be nonzero
for $0 \le t \le s+t$, i.e., for $s\ge0$ and $t\ge0$.

Equivalently, $H^i_{mot}(\F, \Q(w))$ can only be nonzero for $w \ge 0$
and $i \le w$.  This does not ensure that $H^i_{mot} = 0$ for $i<0$!
\end{remark}

\section{Mixed motives}

It is expected that there exists a category $MM$ of \emph{mixed motives},
such that the motivic cohomology groups are given by the $\Ext$-groups
$$
H^i_{mot}(\F; \Q(w)) \cong \Ext^i_{MM}(\Q, \Q(w))
$$
classifying $i$-fold extensions from (the pure motive associated to)
$\Q(w)$ to $\Q$ in this category.

If this is true, $H^i_{mot}(\F; \Q(w)) = 0$ for $i<0$, which is equivalent
to $K_{2w-i}(\F)_\Q^{(w)} = 0$ for $i<0$, hence also to
$K_j(\F)_\Q^{(w)} = 0$ for $2w < j$.

Furthermore,
$H^0_{mot}(\F; \Q(w)) \cong \Hom_{MM}(\Q, \Q(w)) = 0$
for $w>0$, so $K_{2w}(\F)_\Q^{(w)} = 0$ for $w>0$, which means that
$K_j(\F)^{(w)}_\Q = 0$ for $2w \le j$ when $w>0$.

These assertions are the content of the Beilinson--Soul{\'e} vanishing
conjecture.

\begin{conjecture}[Beilinson--Soul{\'e}]
$K_j(\F)_\Q^{(w)} = 0$ for $w < j/2$ (and for $w \le j/2$ when $j>0$).
\end{conjecture}

This is equivalent to the rational version of the conjecture as
first stated.

\section{Higher Chow groups}

A construction of integral motivic cohomology groups is given by
Bloch's higher Chow groups.

\begin{definition}
For each $q\ge0$ let
$$
\Delta^q_\F = \Spec \F[x_0, \dots, x_q]/(x_0 + \dots + x_q = 1)
$$
be the affine $q$-simplex over $\F$.  It is isomorphic to $\A^q_\F =
\Spec \F[x_1, \dots, x_q]$, but the $\Delta^q_\F$ combine more naturally
to a precosimplicial variety:
$$
\xymatrix{
\Delta^0_\F \ar@<0.8ex>[r]^{d_0} \ar@<-0.8ex>[r]_{d_1}
& \Delta^1_\F \ar@<1.6ex>[r]^{d_0} \ar[r]|{d_1} \ar@<-1.6ex>[r]_{d_2}
& \Delta^2_\F 
& \dots
}
$$
Let
$$
z^p(\F, q) = \{\text{codimension $p$ cycles $V \subset \Delta^q_\F$
	meeting each face $\Delta^a_\F \to \Delta^q_\F$ transversely}\} \,.
$$
A cycle is an integral sum of irreducible subvarieties.
Pullback of $V$ along the cofaces $d_i \: \Delta^{q-1}_\F \to \Delta^q_\F$
defines face operators $d_i \: z^p(\F, q) \to z^{p-1}(\F, q)$
that assemble to a presimplicial abelian group
$$
\xymatrix{
z^p(\F, 0)
& z^p(\F, 1) \ar@<-0.8ex>[l]_{d_0} \ar@<0.8ex>[l]^{d_1}
& z^p(\F, 2) \ar@<-1.6ex>[l]_{d_0} \ar[l]|{d_1} \ar@<1.6ex>[l]^{d_2}
& \dots
}
$$
There is an associated chain complex $(z^p(\F, *), \partial)$
$$
0 \longleftarrow z^p(\F, 0) \overset{\partial_1}\longleftarrow
z^p(\F, 1) \overset{\partial_2}\longleftarrow
z^p(\F, 2) \longleftarrow \dots
$$
with $\partial_1 = d_0 - d_1$, $\partial_2 = d_0 - d_1 + d_2$, etc.

Bloch's higher Chow groups are the homology groups
$$
CH^p(\F, q) = \frac{ \ker \partial_q }{ \im \partial_{q+1} }
	= H_q( z^p(\F, *), \partial )
$$
of this chain complex.
\end{definition}

(In what generality is $CH^p(X,0) = CH^p(X)$?)

\begin{definition}
Integral motivic cohomology groups can be defined as
$$
H^i_{mot}(\F; \Z(w)) = CH^w(\F, 2w-i) \,.
$$
\end{definition}

\begin{remark}
These give an integral motivic spectral sequence converging to $K_*(\F)$.

Rationally they agree with the weight eigenspace definition of rational
motivic cohomology.

There are no codimension $p$ subvarieties in $\Delta^q_\F$ for $q < p$,
so $z^p(\F, q) = 0$ and $CH^p(\F, q) = 0$ for $q < p$.  Hence
$H^i_{mot}(\F, \Z(w)) = 0$ for $2w - i < w$, i.e., for $i > w$.

This definition does not tell us whether $H^i_{mot} = 0$ for $i < 0$,
or equivalently, if $CH^p(\F, q) = 0$ for $2p < q$.
\end{remark}

\begin{theorem}[Suslin]
With finite coefficients,
$$
CH^p(\F, q; \Z/m) \cong H^{2p-q}_{et}(\F; \Z/m(p))
$$
when $q \ge p$, i.e., when $2p-q \le p$.  In particular, $CH^p(\F, q;
\Z/m) = 0$ for $2p < q$, so $H^i_{mot}(\F, \Z/m(w)) = 0$ for $i < 0$.
\end{theorem}

\begin{conjecture}[Beilinson/Lichtenbaum]
There are complexes
(of Zariski/{\'e}tale sheaves)
$$
\dots \leftarrow \Gamma(w, \F)^w \overset{\delta}\leftarrow
\dots \overset{\delta}\leftarrow
\Gamma(w, \F)^0 \leftarrow \dots
$$
with cohomology calculating motivic cohomology
$$
H^i(\Gamma(w, \F)^*, \delta) \cong H^i_{mot}(\F, \Z(w)) \,.
$$
\end{conjecture}

\begin{remark}
We can let $\Gamma(0, \F) = \Z$ and $\Gamma(1, \F) = \F^\times$
(in cohomological degree~$1$). Lichtenbaum has a proposed complex
$\Gamma(2, F)$.  Goncharov has proposed complexes $\Gamma_{pol}(w, \F)$
associated to polylogarithms, i.e., functions like
$$
Li_w(z) = \sum_{n=1}^\infty \frac{z^n}{n^w} \,.
$$
Here $Li_1(z) = - \ln(1-z)$.
\end{remark}

\section{Quillen's rank filtration}

Recall that $K(\F)_1 = |i S_\bullet \P(\F)|$ is the classifying space
of the simplicial category with $q$-simplices having objects
$$
\sigma \: 0 = V_0 \cof V_1 \cof \dots \cof V_q \,.
$$

\begin{definition}
Let $F_r K(\F)_1 \subset K(\F)_1$ be the subspace consisting of
simplices where $\dim V_q \le r$ (so that $\dim V_i \le r$ for all $i$).
$$
* = F_0 K(\F)_1 \subset \dots \subset F_{r-1} K(\F)_1
	\subset F_r K(\F)_1 \subset \dots \subset K(\F)_1 \,.
$$
\end{definition}

\begin{proposition}
$$
F_r K(\F)_1 / F_{r-1} K(\F)_1 \simeq \Sigma^2 B(\F^r)_{hGL_r(\F)}
	= EGL_r(\F)_+ \wedge_{GL_r(\F)} \Sigma^2 B(\F^r)
$$
is the (based) homotopy orbit space for $GL_r(\F)$ acting on the
double suspension of the Tits building $B(\F^r)$.
\end{proposition}

\begin{proof}[Sketch proof]
The (non-basepoint) $q$-simplices of $F_r K(\F)_1 / F_{r-1} K(\F)_1$
are generated from the objects
$$
\sigma \: 0 = V_0 \cof V_1 \cof \dots \cof V_q
$$
with $\dim V_q = r$, together with the $GL_r(\F)$-action on the
latter.

This is equivalent to the $GL_r(\F)$-homotopy orbits of the subspace with
$q$-simplices
$$
\sigma \: 0 = V_0 \subset V_1 \subset \dots \subset V_{q-1}
	\subset V_q = \F^r \,.
$$
The $i$-th face operator deletes $V_i$.  The $0$-th and $q$-th face operators
map to the base point if $0 \ne V_1$ or $V_{q-1} \ne \F^r$, respectively.

This is equivalent to $\Sigma^2$ of the simplicial set with $(q-2)$-simplices
the chains
$$
0 \subsetneq V_1 \subset \dots \subset V_{q-1} \subsetneq \F^r \,,
$$
which is the nerve of the set of proper, nontrivial subspaces $V$ of
$\F^r$, partially ordered by inclusion, i.e., the Tits building $B(\F^r)$.
An element $A \in GL_r(\F)$ acts on the partially ordered set by mapping
$V$ to $A(V)$, and has the induced action on $B(\F^r)$.
\end{proof}

\begin{example}
$\Sigma^2 B(\F^1) \cong \Delta^1/\partial\Delta^1 \cong S^1$.
\end{example}

\begin{theorem}[Solomon--Tits]
$$
B(\F^r) \simeq \bigvee_\alpha S^{r-2} \,.
$$
\end{theorem}

\begin{definition}
$$
\St_r(\F) = \tilde H_{r-2} B(\F^r) \cong \tilde H_r \Sigma^2 B(\F^r)
	\cong \bigoplus_\alpha \Z
$$
is the \emph{Steinberg representation} of $GL_r(\F)$.
\end{definition}

\begin{corollary}
The homology
$$
\tilde H_*( F_r K(\F)_1 / F_{r-1} K(\F)_1 )
	\cong \tilde H_* (\Sigma^2 B(\F^r)_{hGL_r(\F)})
	\cong H^{gp}_{*-r} ( GL_r(\F); \St_r(\F))
$$
is concentrated in degrees $*\ge r$.  Hence
$$
H_{j+1}(F_r K(\F)_1) \to H_{j+1} K(\F)_1
$$
is surjective for $j+1 = r$, and an isomorphism for $j+1 < r$.
Thus
$$
F_r H_{j+1} K(\F)_1 =
	\im ( H_{j+1}(F_r K(\F)_1) \to H_{j+1} K(\F)_1 )
$$
is equal to $H_{j+1} K(\F)_1$ for $r \ge j+1$.
\end{corollary}

This enters in the proof of the following theorem.

\begin{theorem}[Quillen]
Let $\O_F$ be the ring of integers in a number field~$F$.
For each $j\ge0$ the group $K_j(\O_F)$ is finitely generated.
\end{theorem}

The connectivity conjecture asserts a stronger convergence result, namely
$F_r K_j(\F) = K_j(\F)$ for $2r \ge j+1$, but for the more powerful
stable rank filtration.

\section{The spectrum level rank filtration}

Also recall that $\K(\F) = \{ n \mapsto K(\F)_n
	= | i S^{(n)}_\bullet \P(\F) | \}$
where $i S^{(n)}_\bullet \P(\F)$ has $q$-simplices the category with
objects
\begin{align*}
\sigma \: [q]^n &\to \P(\F) \\
(i_1, \dots, i_n) &\mapsto V_{i_1, \dots, i_n}
\end{align*}
plus choices of subquotients, subject to lists of conditions.

\begin{definition}[Rognes (1992)]
Let $F_r K(\F)_n \subset K(\F)_n$ be the subspace where $\dim V_{q, \dots, q}
\le r$ (so that $\dim V_{i_1, \dots, i_n} \le r$ for all $(i_1, \dots, i_n)$).
Let
$$
F_r \K(\F) = \{n \mapsto F_r K(\F)_n\}
$$
be the associated (pre-)spectrum.  The sequence
$$
* \cof F_1 \K(\F) \cof \dots \cof F_{r-1} \K(\F) \cof F_r \K(\F) \cof
	\dots \cof \K(\F)
$$
is the \emph{spectrum level rank filtration}.
\end{definition}

Recall that $\pi_j \X = \colim_n \pi_{j+n} X_n$ for a prespectrum
$\X = \{n \mapsto X_n\}$.

\begin{definition}
Let
$$
F_r K_j(\F) = \im ( \pi_j F_r \K(\F) \to \pi_j \K(\F) )
$$
so that
$$
0 \subset F_1 K_j(\F) \subset \dots
	\subset F_r K_j(\F) \subset \dots \subset K_j(\F) \,.
$$
This is the \emph{stable rank filtration}.
\end{definition}

\begin{proposition}
$$
F_r \K(\F) / F_{r-1} \K(\F) \simeq \D(\F^r)_{hGL_r(\F)}
	= EGL_r(\F)_+ \wedge_{GL_r(\F)} \D(\F^r)
$$
is the homotopy orbit spectrum for $GL_r(\F)$ acting on the
\emph{stable building} $\D(\F^r)$.
\end{proposition}

\begin{proof}[Sketch proof]
At level~$n$, $F_r K(\F)_n / F_{r-1} K(\F)_n$ realizes a simplicial
category with $q$-simplices diagrams
$$
\sigma \: (i_1, \dots, i_n) \mapsto V_{i_1, \dots, i_n}
$$
with $\dim V_{q, \dots, q} = r$.  It is equivalent to the subcategory
where $V_{q, \dots, q} = \F^r$ and each $V_{i_1, \dots, i_n}$ is a
subspace of $\F^r$, with morphisms given by the $GL_r(\F)$-action on $\F^r$
and its subspaces.
\end{proof}

\begin{definition}
We define $\D(\F^r) = \{n \mapsto D(\F^r)_n\}$ by letting $D(\F^r)_n$
be a simplicial set with $q$-simplices diagrams $\sigma \: [q]^n \to
\Sub(\F^r) \subset \P(\F)$ consisting of subspaces $V_{i_1, \dots, i_n}$
of $\F^r$ and inclusions between these.  The case $n=2$ appears as follows:
$$
\xymatrix{
0 \ar@{}[r]|{=} \ar@{}[d]|{||}
& 0 \ar@{}[r]|{=} \ar@{}[d]|{\cap}
& \dots \ar@{}[r]|{=}
& 0 \ar@{}[d]|{\cap} \\
0 \ar@{}[r]|{\subset} \ar@{}[d]|{||}
& V_{1,1} \ar@{}[r]|{\subset} \ar@{}[d]|{\cap}
& \dots \ar@{}[r]|{\subset}
& V_{1,q} \ar@{}[d]|{\cap} \\
\vdots \ar@{}[d]|{||}
& \vdots \ar@{}[d]|{\cap}
& \ddots
& \vdots \ar@{}[d]|{\cap} \\
0 \ar@{}[r]|{\subset}
& V_{q,1} \ar@{}[r]|{\subset}
& \dots \ar@{}[r]|{\subset}
& V_{q,q}
}
$$
with $\sigma \: (i,j) \mapsto V_{i,j}$.
In general we require (0) that $V_{i_1, \dots, i_n} = 0$ if some $i_s = 0$,
and $V_{q, \dots, q} = \F^r$, (1) that
$V_{i_1, \dots, i_s-1, \dots, i_n} \subset V_{i_1, \dots, i_s, \dots, i_n}$
is an inclusion, (2) that the pushout morphism
$$
V_{\dots, i_s-1, \dots, i_t, \dots} \oplus_{V_{\dots, i_s-1, \dots, i_t-1, \dots}}
V_{\dots, i_s, \dots, i_t-1, \dots} \to
V_{\dots, i_s, \dots, i_t, \dots}
$$
is injective, etc. (to $(n)$).
We call these the \emph{lattice conditions}.
\end{definition}

\begin{example}
$\D(\F^1) \cong \S$ (the sphere spectrum), so
$F_1 \K(\F) \simeq \S_{hGL_1(\F)} = \Sigma^\infty (B\F^\times)_+$.
Rationally, $\pi_j F_1 \K(\F) \cong \pi_j^S(B\F^\times_+)$
is isomorphic to $H_j(B\F^\times) = H^{gp}_j(\F^\times)$, which
is also rationally isomorphic to $\Lambda^j \F^\times$.  Hence
$F_1 K_j(\F) \subset K_j(\F)$ agrees rationally with the image
of Milnor $K$-theory:
$$
F_1 K_j(\F)_\Q = K^M_j(\F)_\Q
$$
as subgroups of $K_j(\F)_\Q$.
\end{example}

\section{The component filtration}

To analyze the stable building $\D(\F^r)$ we associate some invariants
to the simplices $\sigma \: [q]^n \to \Sub(\F^r)$.

\begin{definition}
The \emph{rank jump} at $\vec p = (i_1, \dots, i_n) \in [q]^n$ is
the dimension of the cokernel of the $n$-cube pushout morphism to
$V_{i_1, \dots, i_n}$, i.e., the alternating sum
$$
\sum_{\epsilon_1, \dots, \epsilon_n \in \{0,1\}}
(-1)^{\epsilon_1 + \dots + \epsilon_n}
	\dim V_{i_1 - \epsilon_1, \dots, i_n - \epsilon_n} \,.
$$
It is non-negative by the lattice conditions, and the sum over all $\vec
p$ of the rank jumps is $r = \dim V_{q,\dots,q}$.  Hence there are $r$
distinguished points $\vec p_1, \dots, \vec p_r \in [q]^n$, counted
with multiplicities, where the rank jumps are positive.
\end{definition}

(The ordering of $\vec p_1, \dots, \vec p_r$ is not well-defined.)

A preordering is a reflexive and transitive relation.  It amounts to a
small category with at most one morphism from $i$ to $j$ for each pair
of objects $(i,j)$.

\begin{definition}
The $r$ distinguished points $\vec p_1, \dots, \vec p_r$ inherit a
preordering from the product partial ordering on $[q]^n$.  Let the
\emph{path component count} of $\sigma$, denoted $c(\sigma)$, be
the number of path components of (the classifying space of the category
associated to) this preordering.  Clearly $1 \le c(\sigma) \le r$.
\end{definition}


Face operators in $D(\F^r)_n$ may merge distinguished points, which in
turn may reduce the path component count.

\begin{definition}
Let $F_c D(\F^r)_n \subset D(\F^r)_n$ be the simplicial subset consisting
of simplices $\sigma$ with path component count $c(\sigma) \le c$.
Let
$$
F_c \D(\F^r) = \{n \mapsto F_c D(\F^r)_n\}
$$
be the associated (pre-)spectrum.  The sequence
$$
* \cof F_1 \D(\F^r) \cof \dots \cof F_{c-1} \D(\F^r) \cof F_c \D(\F^r)
\cof \dots \cof F_r \D(\F^r) = \D(\F^r)
$$
is the \emph{component filtration} of the stable building $\D(\F^r)$.
\end{definition}

\begin{example}
$F_1 \D(\F^r) \simeq \Sigma^\infty \Sigma B(\F^r) \simeq \bigvee_\alpha
\S^{r-1}$.
\end{example}

\begin{theorem}
$$
F_c \D(\F^r) / F_{c-1} \D(\F^r) \simeq \bigvee_\beta \S^{r+c-2}
$$
for $1 \le c \le r$.
\end{theorem}

\begin{proof}[Sketch proof]
There is a finer filtration of $\D(\F^r)$ (than the component filtration)
given by restricting the (isomorphism classes of) preorders on $\{1,
\dots, r\}$ given by setting $s \preceq t$ if $\vec p_s \le \vec p_t$.
The filtration subquotients of this \emph{preorder filtration} can be completely
analyzed, in terms of configuration spaces and smash products of Tits
buildings.  The preorders that are not componentwise (pre-)linear
contribute stably trivial filtration subquotients.  The stable homology of
configuration spaces contributes Lie representations, and the
smash products of Tits building contribute tensor products of
Steinberg representations.  See [Rognes (1992)] for details.
\end{proof}

Hence $H_* \D(\F^r)$ is the homology of a free chain complex
$$
0 \to Z_{2r-2} \overset{\partial}\to \dots
	\overset{\partial}\to Z_{r-1} \to 0 \,,
$$
with
$$
Z_{r+c-2} = H_{r+c-2}( F_c \D(\F^r) / F_{c-1} \D(\F^r) )
	\quad (\cong \bigoplus_\beta \Z)
$$
for $1 \le c \le r$.  In particular,
$$
Z_{2r-2} = \Z[GL_r(\F)/T_r] \otimes_{\Sigma_r} \Lie_r^*
$$
and $Z_{r-1} = \St_r(\F)$.
Here $T_r \subset GL_r(\F)$ is the diagonal torus, and $\Lie_r^*$ is
the dual of the Lie representation of the symmetric group $\Sigma_r$.
The group $GL_r(\F)$ acts naturally on this complex.

\begin{corollary}
$H_* \D(\F^r)$ is concentrated in the range $r-1 \le * \le 2r-2$.
\end{corollary}

\section{The connectivity conjecture}

In [Rognes (1992)] we made the following conjecture.

\begin{conjecture}[Connectivity]
$H_* \D(\F^r)$ is concentrated in degree~$(2r-2)$.
\end{conjecture}

Equivalently, the complex
$$
0 \to H_{2r-2} \D(\F^r) \to Z_{2r-2} \overset{\partial}\to \dots
\overset{\partial}\to Z_{r-1} \to 0
$$
is exact, $\D(\F^r)$ is $(2r-3)$-connected, and $\D(\F^r) \simeq
\bigvee_\gamma \S^{2r-2}$.

\begin{theorem}[Rognes]
The connectivity conjecture is true for $r=1$, $2$ and~$3$.
\end{theorem}

\begin{definition}
Let
$$
\Delta_r(\F) = H_{2r-2} \D(\F^r)
	\quad (\cong \bigoplus_\gamma \Z)
$$
be the \emph{stable Steinberg representation} of $GL_r(\F)$.
\end{definition}

\begin{example}
$\Delta_1(\F) = \Z$ and $\Delta_2(\F)$ is $H_1$ of the
complete graph on the set $\bbP^1(\F)$ of lines $L \subset \F^2$.
\end{example}

\begin{corollary}
If the connectivity conjecture holds, then
$$
H_*( F_r \K(\F) / F_{r-1} \K(\F)) \cong H_* (\D(\F^r)_{hGL_r(\F)})
\cong H_{*-2r+2}^{gp}(GL_r(\F); \Delta_r(\F))
$$
is concentrated in degrees~$* \ge 2r-2$.
Then $F_r \K(\F) \to \K(\F)$ is $(2r-1)$-connected, so
$$
F_r K_j(\F) = \im (\pi_j F_r \K(\F) \to \pi_j \K(\F))
$$
is equal to $K_j(\F)$ for $j \le 2r-1$, or equivalently, for
$2r \ge j+1$.
\end{corollary}

\begin{remark}
For $r\ge2$, if $H_0^{gp}(GL_r(\F); \Delta_r(\F)) =
\Delta_r(\F)_{GL_r(\F)}$ is torsion, hence rationally trivial, then $F_r
K_j(\F)_\Q = K_j(\F)_\Q$ also for $j = 2r$, i.e., for $2r \ge j$.
\end{remark}

\section{The stable rank conjecture}

Applying homology to the sequence of homotopy cofiber sequences
$$
\xymatrix{
{*} \ar[r] & F_1 \K(\F) \ar[d]^{\simeq} \ar[r]
& F_2 \K(\F) \ar[d] \ar[r]
& \dots \ar[r]
& F_r \K(\F) \ar[d] \ar[r]
& \dots \ar[r]
& \K(\F) \\
& \Sigma^\infty B\F^\times_+ 
& \D(\F^2)_{hGL_2(\F)}
& 
& \D(\F^r)_{hGL_r(\F)}
}
$$
with $F_r \K(\F)$ in filtration $s = r-1$
we obtain the \emph{homological rank spectral sequence}
$$
E^1_{s,t}(rk) = H_{s+t}(\D(\F^{s+1})_{hGL_{s+1}(\F)})
	\Longrightarrow_s H_{s+t} \K(\F) \,.
$$
It is of homological type, concentrated in the first quadrant
($s\ge0$ and $t\ge0$).
Assuming the connectivity conjecture, the $E^1$-term can be rewritten
as
$$
E^1_{s,t}(rk) = H_{t-s}^{gp}(GL_{s+1}(\F); \Delta_{s+1}(\F)) \,,
$$
hence is in fact concentrated in the wedge $s\ge0$ and $t\ge s$.
The Hurewicz homomorphism $K_{s+t}(\F) = \pi_{s+t} \K(\F) \to H_{s+t}
\K(\F)$ is a rational equivalence, so after rationalization the rank
spectral sequence converges to $K_{s+t}(\F)_\Q$.

$$
\xymatrix{
t & \dots & \dots & \dots & \dots & \dots \\
4 & H^{gp}_4(\F^\times)
& H_3^{gp}(GL_2 \F; \Delta_2 \F) \ar[l]
& H_2^{gp}(GL_3 \F; \Delta_3 \F) \ar[l]
& H_1^{gp}(GL_4 \F; \Delta_4 \F) \ar[l] & \dots \ar[l] \\
3 & H^{gp}_3(\F^\times)
& H_2^{gp}(GL_2 \F; \Delta_2 \F) \ar[l]
& H_1^{gp}(GL_3 \F; \Delta_3 \F) \ar[l]
& \Delta_4(\F)_{GL_4 \F} \ar[l] & \dots \\
2 & H^{gp}_2(\F^\times)
& H_1^{gp}(GL_2 \F; \Delta_2 \F) \ar[l]
& \Delta_3(\F)_{GL_3 \F} \ar[l] 
& 0 & \dots \\
1 & \F^\times
& \Delta_2(\F)_{GL_2 \F} \ar[l]_{d^1} 
& 0
& 0 & \dots \\
0 & \Z
& 0
& 0
& 0 & \dots \\
E^1_{s,t}(rk) & 0 & 1 & 2 & 3 & s
}
$$

\begin{example}
$E^1_{0,t}(rk) = H_t(B\F^\times) = H^{gp}_t(\F^\times)$ is rationally
isomorphic to $\Lambda^t \F^\times$.
\end{example}

The $E^1$-term suggests the following definition of the motivic complexes
sought by Beilinson and Lichtenbaum.

\begin{definition}
For each $w\ge0$ define the \emph{rank complex}
$(\Gamma_{rk}(w, \F)^*, \delta)$ by
$$
\Gamma_{rk}(w, \F)^i = E^1_{w-i,w}(rk)
$$
and $\delta^i = d^1_{w-i,w} \: \Gamma_{rk}(w, \F)^i \to \Gamma_{rk}(w,
\F)^{i+1}$.
\end{definition}

By definition, $\Gamma_{rk}(w, \F)^i = 0$ for $i > w$.
If the connectivity conjecture holds, then
$$
\Gamma_{rk}(w, \F)^i \cong H^{gp}_i(GL_{w-i+1}(\F); \Delta_{w-i+1}(\F))
$$
is nonzero only for $0 \le i \le w$.

\begin{definition}
Let the \emph{rank cohomology} $H^*_{rk}(\F; \Z(w))$ be the cohomology
of this cochain complex:
$$
H^i_{rk}(\F; \Z(w)) = \frac{\ker \delta^i}{\im \delta^{i-1}}
	= H^i(\Gamma_{rk}(w, \F)^*, \delta) \,.
$$
These groups give the $E^2$-term of the homological rank spectral
sequence
$$
E^2_{s,t}(rk) = H^{t-s}_{rk}(\F; \Z(t))
	\Longrightarrow_s H_{s+t} \K(\F) \,.
$$
\end{definition}

If the connectivity conjecture holds, then this spectral sequence
is concentrated in the region $0 \le s \le t$ (with $s < t$ for $t>0$
if $\Delta_r(\F)_{GL_r(\F)}$ is torsion).

\begin{conjecture}[Stable Rank]
The motivic spectral sequence and the stable rank spectral sequence
are rationally isomorphic, starting from the $E^2$-terms:
$$
\xymatrix{
E^2_{s,t}(mot)_\Q = H^{t-s}_{mot}(\F, \Z(t))_\Q
	\ar@{=>}[r] \ar[d]_{\cong}^{(?)}
	& K_{s+t}(\F)_\Q \ar[d]^{\cong} \\
E^2_{s,t}(rk)_\Q = H^{t-s}_{rk}(\F, \Z(t))_\Q
	\ar@{=>}[r]
	& H_{s+t} \K(\F)_\Q
}
$$
\end{conjecture}

\begin{theorem}
The stable rank conjecture holds for $s=0$.
More precisely,
$$
\xymatrix{
E^2_{0,j}(mot) \ar@{->>}[r] \ar[d] & E^\rho_{0,j}(mot) \ar[d] \\
E^2_{0,j}(rk) \ar@{->>}[r] & E^\rho_{0,j}(rk)
}
$$
consists of rational isomorphisms for all $\rho\ge2$.
\end{theorem}

\begin{proof}[Sketch proof]
Consider the diagram
$$
\xymatrix{
\Lambda^j \F^\times \ar@{->>}[r] \ar[d]
	& K^M_j(\F) \ar[r] \ar[d]
	& K_j(\F) \ar[d] \\
H^{gp}_j \F^\times \ar[r]
	& E^2_{0,j}(rk) \ar[r]
	& H_j \K(\F) \,.
}
$$
\end{proof}

The connectivity and stable rank conjectures together imply (the rational
form of) the Beilinson--Soul{\'e} vanishing conjecture.

An advantage of the stable rank point of view is that $\D(\F^r)$
is described only in terms of linear subspaces $V \subset \F^r$,
as opposed to general subvarieties of $\Delta^q_\F$.

\section{The common basis complex}

By covering the stable building $\D(\F^r)$ by the $GL_r(\F)$-translates
of a \emph{stable apartment} $\bfA(r)$, we obtain the following
elementary description of the stable building.

\begin{definition}
Let the \emph{common basis complex} $D'(\F^r)$ be the simplicial
complex with vertices the proper, nontrivial subspaces $0 \subsetneq V
\subsetneq \F^r$, such that a set $\{V_0, \dots, V_p\}$ of vertices spans
a $p$-simplex if and only if these vector spaces admit a \emph{common
basis}, i.e., there exists a basis $\B = \{b_1, \dots, b_r\}$ for $\F^r$
such that for each $i = 0, \dots, p$ there is a subset of $\B$ that is
a basis for $V_i$.
\end{definition}

\begin{theorem}
$\Sigma^\infty \Sigma D'(\F^r) \simeq \D(\F^r)$.
\end{theorem}

\begin{proof}[Sketch proof]
The stable apartment $\bfA(r)$ the a (pre-)spectrum with $n$-th space
$A(r)_n$ a simplicial set with $q$-simplices diagrams $\sigma \: [q]^n
\to \Sub(\{1, \dots, r\})$ consisting of subsets of $\{1, \dots, r\}$
and inclusions between these.  We know that $A(r)_n \cong S^{rn}$,
so $\bfA(1) \cong \S$ and $\bfA(r) \simeq *$ for $r\ge2$.  (This,
incidentally, gives a proof of the Barratt--Priddy--Quillen theorem.)

The free $\F$-vector space functor $\Sub(\{1, \dots, r\}) \to
\Sub(\F^r)$ induces an embedding $\bfA(r) \cof \D(\F^r)$, and the
translates $\{g \bfA(r) \mid g \in GL_r(\F)\}$ cover $\D(\F^r)$.
A $(p+1)$-fold intersection
$$
g_0 \bfA(r) \cap \dots \cap g_p \bfA(r)
$$
is isomorphic to $\S$ if there is a proper, nontrivial subspace $V
\subset \F^r$ such that for each $0 \le s \le p$ there is a basis for
$V$ given by a subset of the columns of $g_s \in GL_r(\F)$.  Otherwise,
the intersection is (stably) contractible.  Hence $\D(\F^r) \simeq
\Sigma^\infty \Sigma D''(\F^r)$, where $D''(\F^r)$ is the simplicial
complex with vertices the elements $g$ of $GL_r(\F)$, such that $\{g_0,
\dots, g_p\}$ span a $p$-simplex if and only if there is a $0 \subsetneq
V \subsetneq \F^r$ such that for each $0 \le s \le p$ a subset
of the columns of $g_s$ is a basis for $V$.

For each $0 \subsetneq V \subsetneq \F^r$ the set of $g \in GL_r(\F)$
such that a subset of the columns of $g$ is a basis for $V$
span a contractible subspace $C(V) \subset D''(\F^r)$.  A
$p$-fold intersection
$$
C(V_0) \cap \dots \cap C(V_p)
$$
is contractible if there exists a single $g \in GL_r(\F)$
such that for each $0 \le t \le p$ a subset
of the columns of $g$ is a basis for $V_t$.  In other words,
the intersection is contractible if $\{V_0, \dots, V_p\}$ admit
a common basis.  Otherwise the intersection is empty.
This proves that $D''(\F^r) \simeq D'(\F^r)$.
\end{proof}

\begin{conjecture}[Connectivity]
$\tilde H_* D'(\F^r)$ is concentrated in degree $(2r-3)$.
\end{conjecture}

\begin{example}
For $r=2$, $D'(\F^2)$ is the complete graph on the set $\bbP^1(\F)$
of lines $L \subset \F^2$.  It is connected, hence its homology
$\tilde H_* D'(\F^2)$ is concentrated in degree~$1$.
Thus $\Delta_2(\F)$ is the homology of the complete graph on $\bbP^1(\F)$,
as previously claimed.
\end{example}


\begin{bibdiv}
\begin{biblist}

\bib{MR1191383}{article}{
   author={Rognes, John},
   title={A spectrum level rank filtration in algebraic $K$-theory},
   journal={Topology},
   volume={31},
   date={1992},
   number={4},
   pages={813--845},
   issn={0040-9383},
   review={\MR{1191383}},
   doi={10.1016/0040-9383(92)90012-7},
}

\end{biblist}
\end{bibdiv}
\end{document}